\documentclass[12pt,draft]{amsart}
\usepackage{amssymb}
\usepackage{latexsym}
\usepackage{amsfonts}
\usepackage{amsmath}

\newtheorem{theorem}{Theorem}[section]
\newtheorem{proposition}[theorem]{Proposition}
\newtheorem{lemma}[theorem]{Lemma}

\newtheorem{problem}[theorem]{Problem}

\theoremstyle{definition}

\newtheorem{example}[theorem]{Example}

\newcommand{\FCR}{FCR}

\newcommand{\K}{\mathcal K}
\newcommand{\T}{\mathcal T}

\newcommand{\defeq}{\mathrel{:\mkern-0.25mu=}}
\newcommand{\sy}[1]{{\sf S}_{#1}}

\newcommand{\sym}[1]{{\sf Sym}\,#1}

\renewcommand{\wr}{\,{\sf wr}\,}

\newcommand{\aut}[1]{{\sf Aut}\,{#1}}

\newcommand{\psl}[2]{\mbox{\sf PSL}_{#1}(#2)}

\newcommand{\F}{\mathbb F}

\renewcommand{\leq}{\leqslant}
\renewcommand{\geq}{\geqslant}

\newcommand{\supp}{{\rm Supp}\,}

\font\tenvr=cmmi10 scaled 1600
\textfont"F=\tenvr
\renewcommand{\wr}
    {\mathrel{\mkern-1mu\mathchar"0F7B\mkern-1mu}}

\begin{document}

\title[Uniform automorphisms and groups of simple diagonal type]
      {Group factorisations, uniform automorphisms,
        and permutation groups of simple diagonal type}
\author{Cheryl E. Praeger}
\address[C. E. Praeger]{School of Mathematics and Statistics\\
The University of Western Australia\\
35 Stirling Highway 6009 Crawley\\
Western Australia\newline
 cheryl.praeger@uwa.edu.au\\
www.maths.uwa.edu.au/$\sim$praeger}
\author{Csaba Schneider}
\address[C. Schneider]{Departamento de Matem\'atica\\
Instituto de Ci\^encias Exatas\\
Universidade Federal de Minas Gerais\\
Av.\ Ant\^onio Carlos 6627\\
Belo Horizonte, MG, Brazil\\
csaba@mat.ufmg.br\\
 www.mat.ufmg.br/$\sim$csaba}

\begin{abstract}
  We present a new proof, which is independent of the finite simple group
  classification and applies also to infinite groups,
  that quasiprimitive permutation groups of simple diagonal type cannot be
  embedded into wreath products in product action.
  The proof uses several deep results that concern factorisations of direct
  products involving subdirect subgroups. We find that such factorisations are
  controlled by the existence of uniform automorphisms.
\end{abstract}
\date{\today}

\thanks{We are grateful to Robert Guralnick and Kay Magaard for
  their help with the  construction
  in Example~\ref{ex:uniaut}. We thank Pavel Shumyatsky for his advice on the
  literature on uniform automorphisms. The first author wishes to acknowledge the support of the Australian Research Council Discovery Grant (DP160102323).
  The  second author was supported by the
 research projects 302660/2013-5 {\em (CNPq, Produtividade em Pesquisa)},
 475399/2013-7 {\em (CNPq, Universal)}, and 
 the APQ-00452-13 {\em (Fapemig, Universal)}.}
\maketitle

\section{Introduction}\label{section:intro}

The O'Nan--Scott Theorem for finite primitive and quasiprimitive permutation
groups identifies several classes of such groups and claims that
each  primitive or quasiprimitive group is a member of a unique class.
In several combinatorial and
group theoretic applications, it is necessary to understand
the possible inclusions among primitive and quasiprimitive permutation groups.
Such inclusions were studied by the first author for finite
primitive permutation
groups in~\cite{prae:incl}, while inclusions of finite quasiprimitive groups
in primitive groups were described by  R.~W.~Baddeley and the first author
in~\cite{bad:quasi}. The possible inclusions are described in both cases by
considering separately each of the O'Nan--Scott types of these
permutation groups.
Some of these results rely on the finite simple group classification (FSGC).
In this paper we show how to remove the assumption of finiteness  and hence the
dependence on the FSGC for embeddings of simple diagonal type groups into
wreath products in product action (see Section~\ref{section:sd} for the
definitions).

It was proved in~\cite{prae:incl}
that a finite primitive group of simple diagonal
type cannot be a subgroup of a wreath product in product action and
a similar result was proved in~\cite[Corollary~1.3]{transcs}
for finite quasiprimitive groups of simple
diagonal type. The latter theorem formed an important part of the description
of primitive overgroups of finite quasiprimitive permutation
groups in~\cite{bad:quasi}.
As with many results concerning the inclusion problem,
\cite[Corollary~1.3]{transcs}  depended 
on the FSGC. As mentioned above, we extend this
result to infinite permutation groups of simple diagonal type and
prove that such groups cannot be embedded into primitive groups of product
action type.

\begin{theorem}\label{mainth}
  Suppose that $G$ is a quasiprimitive permutation group of simple diagonal
  type acting on a possibly infinite set $\Omega$. Then $G$ is not contained
  in a subgroup of $\sym\Omega$ which is permutationally isomorphic to
  $\sym\Gamma\wr\sy\ell$ with $|\Gamma|\geq 2$ and $\ell\geq 2$
  acting in its product action.
\end{theorem}

As is typical in the study of the inclusion problem,
the key result underpinning the proof of Theorem~\ref{mainth} in the finite case is a factorisation result~\cite[Lemma~2.2]{bad:quasi}
concerning factorisations
of finite characteristically simple groups using subdirect subgroups as factors.
The proof of~\cite[Lemma~2.2]{bad:quasi} depends on the fact that a finite simple group does
not admit uniform automorphisms, which, in turn, is one of the well-known
consequences of the FSGC. We noticed that this
result can be extended to arbitrary groups that do not admit a uniform
automorphism; see Section~\ref{set:factstrip} for the definition of a uniform automorphism and of a strip.

\begin{theorem}\label{nostripfact}
  Suppose that $T$ is a group that does not admit a uniform
  automorphism and that $X,\ Y$ are direct products of pairwise
  disjoint  non-trivial full
  strips in $T^k$ with $k\geq 2$. Then $XY\neq T^k$.
\end{theorem}

The proof of Theorem~\ref{mainth} relies on the technical
Proposition~\ref{mainstripfact}, and, in turn, the proof
of Proposition~\ref{mainstripfact} uses Theorem~\ref{nostripfact}.

The O'Nan--Scott Theorem for finite primitive
groups was originally intended to describe the maximal
subgroups of finite symmetric groups; see~\cite{sco:rep,lps:maxsub}. Along the same lines,
some maximal subgroups of infinite symmetric groups have been associated
with structures such as subsets, partitions~\cite{braziletal},
and cartesian decompositions~\cite{covmpmek}. One class of maximal
subgroups of a finite symmetric group is formed by maximal subgroups of
simple diagonal type. It would be interesting to explore
if  infinite symmetric groups also have maximal subgroups that are associated
(via, for instance, filters or ideals, as in~\cite{braziletal,covmpmek})
to groups of simple diagonal type.
Theorem~\ref{mainth} suggests
that none of the maximal subgroups considered in~\cite{braziletal,covmpmek} contain a primitive
group of simple diagonal type. On the other hand,
simple diagonal type groups do lie in maximal subgroups of $\sym\Omega$,
at least when $\Omega$ is countable by~\cite[Theorem~1.1]{macpr}.

It is a curious fact that the results presented in this note have two
distinct
points of
connection with the work of our late friend and colleague, Laci Kov\'acs.
As early as the 1960's Laci studied uniform automorphisms
for solvable groups (see for instance~\cite{uniform}), and in the 1980's, his attention
turned towards the theory of primitive permutation groups. He devoted
an entire paper~\cite{kov:sd} to primitive groups of simple diagonal type, and it was
this paper in which the terminology `simple diagonal type' was first used.

\section{Uniform automorphisms and factorisations of direct products}
\label{set:factstrip}

In this section we prove several results on factorisations
of direct products using  diagonal subgroups.

An automorphism $\alpha$ of a group $G$ is 
{\em uniform}, 
if the map $\tilde\alpha:g\mapsto g^{-1}(g\alpha)$ is surjective.
If $G$ is a finite group, then a map $G\rightarrow G$ is surjective 
if and only if it is injective. Thus, 
if $G$ is 
a finite group, then $\alpha\in\aut G$ is not uniform
if and only if the map $g\mapsto g^{-1}(g\alpha)$ is not injective; 
that is, $g^{-1}(g\alpha) = h^{-1}(h\alpha)$ for some distinct $g,\ h\in G$. The last equation 
is equivalent to $hg^{-1}=(hg^{-1})\alpha$; that is, in this case,
the element
$hg^{-1}$ is a non-identity 
fixed-point of the automorphism $\alpha$. It is a consequence of
the FSGC
that every automorphism of a non-abelian
finite simple group has non-identity fixed 
points~\cite[Theorem~1.48]{gorenst}. In fact the  following stronger
result holds; see~\cite[9.5.3]{ks} for a proof.

\begin{lemma}\label{fsgorth}
  A  finite non-solvable  group has no uniform (that is, fixed-point-free)
  automorphisms.
\end{lemma}

  The situation is rather different for the class of
  infinite simple groups.

\begin{example}\label{ex:uniaut}
 Let $\F$ be the
algebraic closure of the finite field $\F_p$ and consider
the group $G={\sf SL}_n(\F)$ with $n\geq 2$. Then $G$ is a connected algebraic group
(see~\cite[Exercise~2.2.2]{springer}). Further, the $p$-th powering map 
$(a_{i,j})\mapsto (a_{i,j}^p)$ defines an automorphism $\varphi:G\rightarrow G$
known as the {\em Frobenius automorphism}.\index{automorphism!Frobenius}
By Lang's Theorem~\cite[Theorem~4.4.17]{springer},
\index{Lang's Theorem}the map
$G\rightarrow G,\ g\mapsto g^{-1}(g\varphi)$ is surjective. Since the centre $\mathsf Z(G)$ is
invariant under automorphisms of $G$, $\varphi$ induces an automorphism
$\overline\varphi$
of the infinite simple group $\psl n\F\cong G/\mathsf Z(G)$ such that the map
$\psl n\F\rightarrow\psl n\F$ defined as 
$g\mapsto g^{-1}(g\overline\varphi)$ is surjective. Thus $\overline\varphi$ is a uniform
automorphism of the infinite simple group $\psl n\F$. Similar examples can be
constructed using other connected algebraic groups of Lie type.
\end{example}

Suppose that $G$ is a direct product $G=G_1\times\cdots\times G_k$
of groups $G_i$
and, for $i=1,\ldots,k$, let $\sigma_i:G\rightarrow G_i$ be
the coordinate projection. A subgroup $H$ of $G$ is said to be a
{\em strip} if, for all $i$, 
$H\sigma_i\cong H$ or $H\sigma_i=1$. If $H$ is a strip, then
$H$ is a diagonal subgroup in the sense that
there exist indices $1\leq i_1<\cdots<i_m\leq k$, a subgroup
$H_0\leq G_{i_1}$ and, for $j\in\{2,\ldots,m\}$, injective homomorphisms
$\alpha_j: H_0\rightarrow G_{i_j}$ such that
\begin{eqnarray}\label{hstrip}
H&=&\{(g_1,\ldots,g_k)\mid g_{i_1}\in H_0,\
g_{i_j}=g_{i_1}\alpha_j\mbox{ for $2\leq j\leq m$ and }\\
\nonumber &&g_i=1\mbox{ if }i\not\in\{i_1,\ldots,i_m\}\}.
\end{eqnarray}
The set $\{G_{i_1},\ldots,G_{i_m}\}$ is said to be the {\em support}
of the strip $H$ and is denoted by $\supp H$. The strip $H$ is said
to be {\em non-trivial}, if $|\supp H|\geq 2$, and it is said
to be {\em full} if $H\sigma_i=G_i$ whenever $G_i\in\supp H$.
Thus, if $H$, written as in~\eqref{hstrip},  is a full strip,
then  $H_0=G_{i_1}$ and the $\alpha_j$ are isomorphisms. In particular,
the $G_{i_j}$ are pairwise isomorphic for $j=1,\ldots,m$.

The existence of factorisations of direct products with 
strips as factors is related to the existence of uniform automorphisms.
It was proved in~\cite[Lemma~2.2]{bad:quasi} that
a finite direct power of a finite simple group cannot be factorised into
a product of two subgroups both of which are direct products of non-trivial
full strips. We generalise this result for a larger class of groups.
We start by proving the following lemma for two factors.

\begin{lemma}\label{orthstrip}
  The following assertions are equivalent for a group $G$.
  \begin{enumerate}
  \item There exist non-trivial full strips $X$ and $Y$
    in $G\times G$
    such that $G\times G= XY$.
  \item $G$ admits a uniform automorphism.
    \end{enumerate}
\end{lemma}
\begin{proof}
  Suppose that $G\times G=XY$ with $X$ and $Y$
  non-trivial full strips of $G\times G$.
  Then there exist $\alpha,\ \beta\in\aut G$ such that
  $X=\{(g,g\alpha)\mid g\in G\}$ and
  $Y=\{(g,g\beta)\mid g\in G\}$. If $g\in G$, then there exist
$g_1,\ g_2\in G$,  such that $(g,1)=(g_1^{-1},g_1^{-1}\alpha)(g_2,g_2\beta)$. 
Thus $g=g_1^{-1}g_2$ and $1=(g_1^{-1}\alpha)(g_2\beta)$, and so 
$g_2=g_1\alpha\beta^{-1}$.  Hence $g=g_1^{-1}(g_1\alpha\beta^{-1})$, which
implies that
the map $x\mapsto x^{-1}(x\alpha\beta^{-1})$ is surjective. That is, the 
automorphism $\alpha\beta^{-1}$ is uniform.

Conversely, assume that $\alpha\in\aut G$ is uniform.
Set $X=\{(g,g)\mid g\in G\}$ and $Y=\{(g,g\alpha)\mid g\in G\}$.
Let $(x,y)\in G\times G$. Choose $h\in G$ such that $h^{-1}(h\alpha)=x^{-1}y$
(such an $h$ exists, as $\alpha$ is uniform)
and let $g=xh^{-1}$. Then $g(h\alpha)=xh^{-1}(h\alpha)=y$ and $gh=x$. Thus
$(g,g)(h,h\alpha)=(x,y)$. Therefore $XY=G\times G$.
\end{proof}

Finite groups with uniform automorphisms do exist. For example, the 
automorphism $\alpha:x\mapsto x^{-1}$ of a finite abelian group $G$ of odd order
is uniform. In this case, we do in fact obtain a  factorisation
$$
G\times G=\{(g,g)\mid g\in G\}\{(g,g\alpha)\mid g\in G\}.
$$


\begin{lemma}\label{doublestrips}
  Suppose that $T$ is a group and, for $i=1,\ldots,d$,
  let $\alpha_i,\ \beta_i\in \aut T$. Consider the following subgroups
  $X$ and $Y$ of $T^{2d}$:
  \begin{eqnarray*}
    X&=&\{(t_1,t_1\alpha_1,t_2,t_2\alpha_2\ldots,t_d,t_d\alpha_d)\mid
    t_i=T\} \mbox{ and }\\
    Y&=&\{(s_d\beta_d,s_1,s_1\beta_1,s_2,s_2\beta_2,\ldots,s_{d-1},s_{d-1}\beta_{d-1},s_d)\mid s_i\in T\}.
\end{eqnarray*}
Then the following are equivalent:
\begin{enumerate}
\item $XY=T^{2d}$;
\item the automorphism $\alpha_1\beta_1\alpha_2\beta_2\cdots\alpha_d\beta_d$ of $T$ is uniform.    
\end{enumerate}
\end{lemma}
\begin{proof}
Assume that $XY=T^{2d}$ and
let $x\in T$. Then there exist
\begin{eqnarray*}
    (t_1^{-1},t_1^{-1}\alpha_1,t_2^{-1},t_2^{-1}\alpha_2,
    \ldots,t_d^{-1},t_d^{-1}\alpha_d)&\in& X\mbox{ and }\\
    (s_d\beta_d,s_1,s_1\beta_1,\ldots,s_{d-1},s_{d-1}\beta_{d-1},s_d)&\in& Y
\end{eqnarray*}
such that $(x,1,\ldots,1)$ is equal to
$$(t_1^{-1},t_1^{-1}\alpha_1,t_2^{-1},t_2^{-1}\alpha_2,\ldots,t_d^{-1},t_d^{-1}\alpha_d)\cdot
(s_d\beta_d,s_1,s_1\beta_1,\ldots,s_{d-1},s_{d-1}\beta_{d-1},s_d).
$$
Comparing the entries for these two expressions, we obtain that
\begin{eqnarray*}
  x&=&t_1^{-1}(s_{d}\beta_d)\\
  t_1&=&s_1\alpha_1^{-1}\\
  t_2&=&s_1\beta_1\\
  t_2&=&s_2\alpha_2^{-1}\\
  &\vdots&\\
  t_d&=&s_{d-1}\beta_{d-1}\\
  t_d&=&s_{d}\alpha_d^{-1}.
\end{eqnarray*}
Hence  $s_{i}=s_{i-1}\beta_{i-1}\alpha_{i}$ for $i=2,\ldots,d$.
Thus, by induction,
$s_d=s_1\beta_1\alpha_2\cdots\beta_{d-1}\alpha_d$.
Then setting $s=s_1\alpha_1^{-1}$ and using the first two equations in the system above,
we obtain 
$$
x=(s_1\alpha_1^{-1})^{-1}(s_1 \beta_1\alpha_2\cdots\beta_{d-1}\alpha_d\beta_d)=
s^{-1}(s\alpha_1\beta_1\ldots\alpha_d\beta_d).
$$
Since we chose
$x\in T$ arbitrarily, the automorphism
$\alpha_1\beta_1\ldots\alpha_d\beta_d$ of $T$ is uniform.

Conversely, assume that $\alpha\defeq\alpha_1\beta_1\ldots\alpha_d\beta_d$ is a 
uniform automorphism of $T$.
Let $x=(x_1,\ldots,x_{2d})$ be an arbitrary
element of $T^{2d}$. Since $\alpha$ is uniform, there exists $s_0\in T$  such that
\begin{equation}\label{seq}
  s_0^{-1}(s_0\alpha)=\prod_{i=d}^1\left(\left(x_{2i}^{-1}\beta_i\alpha_{i+1}\cdots\alpha_d\beta_d\right)
  \left(x_{2i-1}\alpha_i\beta_i\cdots\alpha_d\beta_d\right)\right).
\end{equation}
Let $s_d=s_0\beta_d^{-1}$ and $t_d=(s_dx_{2d}^{-1})\alpha_d^{-1}$. We next
define sequence 
$$
s_{d-1},\ t_{d-1},\ s_{d-2},\ldots,s_1,\ t_1
$$
`backwards recursively' in the sense that we define,
for $i=d-1,\ldots,1$, the elements $s_i$ and $t_i$
assuming that we have defined the elements $s_{i+1}$ and $t_{i+1}$:
\begin{eqnarray*}
  s_i&=&(t_{i+1}x_{2i+1})\beta_i^{-1};\\
    t_i&=&(s_ix_{2i}^{-1})\alpha_i^{-1}.
\end{eqnarray*}
Now let
\begin{eqnarray*}
t&=&(t_1^{-1},t_1^{-1}\alpha_1,\ldots,t_{d}^{-1},t_d^{-1}\alpha_d)\quad\mbox{and}\\
s&=&(s_d\beta_d,s_1,s_1\beta_1,\ldots,s_{d-1},s_{d-1}\beta_{d-1},s_d).
\end{eqnarray*}
We claim that $ts=x$. The definition of $t_i$ and $s_i$ implies that
\begin{eqnarray*}
  x_{2i}&=&(t_i^{-1}\alpha_i)s_i\quad\mbox{and }\\
  x_{2i+1}&=&t_{i+1}^{-1}(s_i\beta_i)
\end{eqnarray*}
for all indices $2i$ and $2i+1$ between $2$ and $2d$. This shows that the
product $ts$ agrees with $x$ from the second coordinate onwards.
It remains to show that $ts$ agrees with $x$ in the first coordinate also.
Using the equations for $t_i$ and  $s_i$, we obtain
by induction that
\begin{multline*}
t_1=(s_d\alpha_d^{-1}\beta_{d-1}^{-1}\cdots\beta_1^{-1}\alpha_1^{-1})
(x_{2d}^{-1}\alpha_d^{-1}\beta_{d-1}^{-1}\cdots\beta_1^{-1}\alpha_1^{-1})\\
(x_{2d-1}\beta_{d-1}^{-1}\cdots\beta_1^{-1}\alpha_1^{-1})\cdots
(x_2^{-1}\alpha_1^{-1}).
\end{multline*}
Applying $\alpha$ to the last equation and considering
the definition of $s_d$, we obtain 
\begin{multline*}
t_1\alpha= (s_d\beta_d)(x_{2d}^{-1}\beta_d)(x_{2d-1}\alpha_d\beta_d)
\cdots (x_2^{-1}\beta_1\alpha_2\beta_2\cdots\alpha_{d}\beta_d)=\\
s_0(x_{2d}^{-1}\beta_d)(x_{2d-1}\alpha_d\beta_d)
\cdots (x_2^{-1}\beta_1\alpha_2\beta_2\cdots\alpha_{d}\beta_d)
\end{multline*}
which gives (using~\eqref{seq} for the last equality in the next line)
\begin{multline*}
s_0^{-1}(t_1\alpha)(x_1\alpha)=
s_0^{-1}(t_1\alpha)(x_1\alpha_1\beta_1\alpha_2\beta_2\cdots\alpha_{d}\beta_d)=\\
(x_{2d}^{-1}\beta_d)(x_{2d-1}\alpha_d\beta_d)
\cdots (x_2^{-1}\beta_1\alpha_2\beta_2\cdots\alpha_{d}\beta_d)
(x_1\alpha_1\beta_1\alpha_2\beta_2\cdots\alpha_{d}\beta_d)=s_0^{-1}(s_0\alpha).
\end{multline*}
This implies that $x_1=t_1^{-1}s_0$ which, in turn, equals $t_1^{-1}(s_d\beta_d)$. Therefore 
$ts$ agrees with $x$ in its first coordinate also, and so
$ts=x$. 
Since the choice of the element $x$ was arbitrary,  $XY=T^{2d}$.
\end{proof}

If $M$ is the direct product $M=T_1\times\cdots\times T_k$ and 
 $I$ is a subset of $\{T_1,\ldots,T_k\}$ or
a subset of $\{1,\ldots,k\}$, then $\sigma_I$ denotes the corresponding
coordinate projection from $M$ to $\prod_{T_i\in I} T_i$ or to
$\prod_{i\in I} T_i$, respectively.

The proof of the
following lemma uses some simple graph theoretic concepts. The graphs that
occur in this proof are undirected graphs without multiple edges and loops.
A graph which does not contain a cycle is said to be a {\em forest},
while a connected graph without a cycle is  a {\em tree}.
The {\em valency} of a vertex $v$ in a graph is the number of vertices
that are adjacent to $v$. 
A vertex $v$ in a forest with valency one is said to be
a {\em leaf}.

We now prove Theorem~\ref{nostripfact}.

\begin{proof}[The proof of Theorem~\ref{nostripfact}]
  By assumption, $X=X_1\times\cdots\times X_r$ and
  $Y=Y_1\times\cdots\times Y_s$, where the $X_i$ and $Y_i$ are
  non-trivial full strips. Suppose that $T^k=XY$. At the end of the
  proof, this will lead to a contradiction.

  Let $\Gamma$ be the
  graph with vertex set $\{X_1,\ldots,X_r,Y_1,\ldots,Y_s\}$
  such that two vertices of $\Gamma$ are adjacent
  if and only if the supports of the strips are not disjoint.
  First note that two
  strips that belong to $X$ are disjoint, and the same is true for
  two strips that belong to $Y$. Hence
  $\Gamma$ is a bipartite graph with bipartition $\{X_1,\ldots,X_r\}\cup
  \{Y_1,\ldots,Y_s\}$.
  We prove the result by proving a series of claims. Recall that 
   $T^k=XY$. Suppose that $T_1,\ldots,T_k$ are the internal direct
  factors of $T^k$; that is $T_1,\ldots,T_k$ are normal subgroups of
  $M$ such that $M=T^k=T_1\times\cdots\times T_k$. 
  
\smallskip

\noindent{\em Claim 1.}
Each vertex of $\Gamma$ lies on at least one edge.

\smallskip

\noindent{\em Proof of Claim 1.}
If, say, $X_1$ lies on no edge of $\Gamma$, then $\supp X_1\cap \supp Y_i=\emptyset$ for each $i$, and so $\supp X_1\cap\supp Y
=\emptyset$. Thus  if $\sigma$ is the projection of $M$ onto
$\prod_{T_i\in\supp X_1}T_i$, then $Y\sigma=1$, and so
$$M\sigma=(XY)\sigma=
X\sigma=X_1\cong T,
$$
contradicting the fact that $X_1$ is non-trivial. Thus
$X_1$ lies on at least one edge and the same proof shows that each $X_i$ and
each $Y_i$ lies on at least one edge. \hfill$\Box$

\smallskip
  
\noindent {\em Claim 2.}
If $X_{j_1}$ and $Y_{j_2}$ are adjacent in $\Gamma$, then
  $|\supp X_{j_1}\cap\supp Y_{j_2}|=1$. 

\smallskip

\noindent{\em Proof of Claim 2.}
This follows from Lemma~\ref{orthstrip}. Indeed, if $T_{i_1},\ T_{i_2}\in
\supp X_{j_1}\cap \supp Y_{j_2}$ and $\sigma=\sigma_{\{i_1,i_2\}}$,
  then $X_{j_1}\sigma$ and $Y_{j_2}\sigma$ are non-trivial full strips in
  $T_{i_1}\times T_{i_2}$ such that $T_{i_1}\times T_{i_2}=(X\sigma)(Y\sigma)=
  (X_{j_1}\sigma)(Y_{j_2}\sigma)$.
  Since $T$ does not admit a uniform automorphism,
  this is a contradiction, by Lemma~\ref{orthstrip}.\hfill$\Box$

  \smallskip

By Claim~2, an edge in $\Gamma$ that connects
$X_{j_1}$ and $Y_{j_2}$ can be labelled with $T_i$ where $\{T_i\}=
\supp X_{j_1}\cap
\supp Y_{j_2}$.

\smallskip
  
  \noindent{\em  Claim 3.}
  $\Gamma$ does not contain a cycle.

  \smallskip

  \noindent{\em Proof of Claim 3.}
  Suppose to the contrary that $\Gamma$ contains a cycle, and choose a cycle $X_1,Y_1,\ldots,X_d,Y_d$ of shortest length $2d$. Then
  $2d\geq 4$ and  the $X_i$ and the $Y_i$ are pairwise distinct.
  By reordering the $T_i$, we may also
  assume without
  loss of generality that
  \begin{eqnarray*}
    \{T_1\}&=&\supp Y_d\cap \supp X_1;\\
    \{T_2\}&=&\supp X_1\cap\supp Y_1;\\
    \{T_3\}&=&\supp Y_1\cap \supp X_2;\\
    &\vdots&\\
    \{T_{2i-1}\}&=&\supp Y_{i-1}\cap\supp X_i;\\
    \{T_{2i}\}&=&\supp X_{i}\cap\supp Y_i;\\
    &\vdots&\\
    \{T_{2d-1}\}&=&\supp Y_{d-1}\cap \supp X_d;\\
    \{T_{2d}\}&=&\supp X_d\cap \supp Y_d.
  \end{eqnarray*}
  In particular, $\supp X_i\cap\{T_1,\ldots,T_{2d}\}=\{T_{2i-1},T_{2i}\}$,
  $\supp Y_i\cap\{T_1,\ldots,T_{2d}\}=\{T_{2i},T_{2i+1}\}$ for $i=1,\ldots,d-1$,
  and $\supp Y_d\cap\{T_1,\ldots,T_{2d}\}=\{T_{2d},T_{1}\}$.
  The factors $T_{2i}$, with $i=1,\ldots,d$, are pairwise distinct, since the $\supp X_i$ are
  pairwise disjoint. Similarly the factors $T_{2i-1}$, with $i=1,\ldots,d$,
  are pairwise distinct since the $\supp Y_i$ are pairwise disjoint.
  Suppose that $T_{2i}=T_{2j-1}$ where $1\leq i,\ j\leq d$.
  Since $T_{2j-1}\in\supp X_j$ and $T_{2i}\in\supp X_i$, we have $i= j$.
  But then $T_{2i-1}\in\supp Y_{i-1}$ while $T_{2i}\in\supp Y_i$ contradicting
  the fact that $\supp Y_{i-1}\cap\supp Y_{i}=\emptyset$.
  Thus all the $T_i$ are pairwise distinct.
Let
  $\sigma$ denote the projection $\sigma_{\{1,\ldots,2d\}}$.
    Then $X\sigma=X_1\sigma\times \cdots \times X_d\sigma$ and
    $Y\sigma=Y_1\sigma\times\cdots\times Y_d\sigma$. Further,
    for $i=1,\ldots,d$, $X_i\sigma$ is a non-trivial full strip in
    $T_{2i-1}\times T_{2i}$, for $i=1,\ldots,d-1$, $Y_i\sigma$ is a
    non-trivial full strip in $T_{2i}\times T_{2i+1}$, while
    $Y_d\sigma$ is a non-trivial full strip in $T_{2d}\times T_1$.
    Thus there exist isomorphisms $\alpha_i:T_{2i-1}\rightarrow T_{2i}$
    (for $i=1,\ldots,d$), $\beta_i:T_{2i}\rightarrow T_{2i+1}$
    (for $i=1,\ldots,d-1)$ and $\beta_d:T_{2d}\rightarrow T_1$
    such that
    \begin{eqnarray*}
    X\sigma&=&\{(t_1,t_1\alpha_1,t_2,t_2\alpha_2,\ldots,t_d,t_d\alpha_d)
    \mid t_i\in T_{2i-1}\}\mbox{ and}\\
    Y\sigma&=&
    \{(s_d\beta_d,s_1,s_1\beta_1,\ldots,s_{d-1},s_{d-1}\beta_{d-1},s_d)
    \mid s_i\in T_{2i}\}.
    \end{eqnarray*}
    Since $XY=T^k$, we find that
    $(X\sigma)(Y\sigma)=T_1\times\cdots\times T_{2d}$.
    By Lemma~\ref{doublestrips}, the automorphism
    $\alpha_1\beta_1\cdots\alpha_d\beta_d$ of $T_1$ is uniform.
    This is a contradiction. Hence  $\Gamma$ does not
contain a cycle.\hfill$\Box$

\smallskip

Since $\Gamma$
does not contain
a cycle, the graph $\Gamma$ is a forest with no-isolated vertices.
This verifies
at once  our next claim

\smallskip

\noindent{\em Claim~4.} There are at least two leaves in $\Gamma$.
\hfill$\Box$

\smallskip

Set $\T=\{T_1,\ldots,T_k\}$,
$$
\mathcal S_1=\bigcup_{i=1}^r\supp X_i\quad\mbox{and}\quad
\mathcal S_2=\bigcup_{i=1}^s\supp Y_i.
$$
For $i=1,\ 2$, let $a_i=|\T\setminus\mathcal S_i|$.
Suppose without loss of generality that $X_{1}$
has
valency one in $\Gamma$. Since $X_1$ is a non-trivial strip,
$|\supp X_1|\geq 2$, and hence there must be some $T_i\in\supp X_1$ that
is not covered by a strip in $Y$. Assume without loss of generality that
$i=1$. This implies that $T_1\in\T\setminus \mathcal S_2$, and so $a_2=|\T\setminus\mathcal S_2|\geq 1$.

\smallskip

\noindent{\em Claim 5.} Some $Y_i$ has valency~one, and hence
$a_1\geq 1$ as well as $a_2\geq 1$.

\smallskip

\noindent{\em Proof of Claim~5.}
Suppose to the contrary that every $Y_i$
has valency at least two. Then a second strip of
$X$, $X_2$ say, also has valency $1$, so  some $T_i\in\supp X_2$
with $i\geq 2$
also lies in $\T\setminus \mathcal S_2$. Let $Y_{s+1}$ be a full strip
such that $\supp Y_{s+1}=\T\setminus\mathcal S_2$.
Since $|\T\setminus\mathcal S_2|\geq 2$, the strip $Y_{s+1}$ is
non-trivial. Set $\bar Y=Y\times Y_{s+1}$.
Since $T^k=XY$, we have  $T^k=X\bar Y$. However, the
graph that corresponds to the factorisation $T^k=X\bar Y$ has
no vertices of valency 1, which contradicts  Claim~4 applied
 to the graph corresponding to the factorisation $T^k=X\bar Y$. 
Thus $Y_i$ has valency one for some $i$, and hence $a_1\geq 1$ also.
\hfill$\Box$

\smallskip

\noindent{\em Claim 6.} $a_1=a_2=1$.

\smallskip

\noindent{\em Proof of Claim 6.}
If $a_1\geq 2$ and $a_2\geq 2$, then let $X_{r+1}$ and $Y_{s+1}$
be full strips such that $\supp X_{r+1}=\T\setminus\mathcal S_1$
and $\supp Y_{s+1}=\T\setminus\mathcal S_2$. Since $a_1,\ a_2\geq 2$,
$X_{r+1}$ and $Y_{s+1}$ are non-trivial strips. Set $\bar X=X\times X_{r+1}$
and $\bar Y=Y\times Y_{s+1}$. Since $T^k=XY$, we have  $T^k=\bar X
\bar Y$. However, the graph that corresponds to the factorisation
$T^k=\bar X\bar Y$ has no vertex of valency one, which contradicts
Claim~4 applied to the graph of the factorisation $T^k=\bar X\bar Y$. 

Thus $\min\{a_1,a_2\}=1$, 
and without loss of generality we may assume that $a_1=1$. 
If $a_2\geq 2$, 
then we may construct
$Y_{s+1}$ and $\bar Y$ as in the proof of  Claim~5. The graph that
corresponds to the  factorisation $T^k=X\bar Y$ has only one vertex of
valency one contradicting Claim 4.
\hfill$\Box$

\smallskip

Let us now obtain a final contradiction.
It follows from Claims~4 and~6 that
there is  
exactly one strip in $X$ with valency~1 and  exactly one such strip in
$Y$. All other strips have valency at least~2.
A forest with precisely two leaves is a path. Hence $\Gamma$ is a path
of the  form
$$
X_1-Y_1-\cdots-X_r-Y_r.
$$
The valencies of $X_1$ and 
of $Y_r$ are equal to one. Further, the valency of each
of $Y_1,X_2,\ldots,X_r$ is equal to two.
Hence $|\supp X_i|=|\supp Y_i|=2$ for all $i$.
Let $$
z=|\{(X_i,T_j)\mid T_j\in\supp X_i\}|.
$$
Since
$|\supp X_i|=2$ for all $i$, we have $z=2r$. On the other hand,
by Claim~6, there exists a unique $j_0$ such that $T_{j_0}\not\in\bigcup_i\supp X_i$
and if $j\neq j_0$ then there is a unique $X_i$ such that $T_j\in\supp X_i$.
Thus $2r=z=k-1$ and so $k=2r+1$. 
By possibly reordering the $T_i$, we may assume that
there exist, for $i=1,\ldots,r$, isomorphisms
$\alpha_i:T_{2i-1}\rightarrow T_{2i}$ and $\beta_i:T_{2i}\rightarrow T_{2i+1}$
such that
\begin{eqnarray*}
X&=&\{(t_1,t_1\alpha_1,t_2,t_2\alpha_2,\ldots,t_r,t_r\alpha_r,1)\mid
t_i\in T_{2i-1}\}\mbox{ and}\\
Y&=&\{(1,s_1,s_1\beta_1,s_2,s_2\beta_2,\ldots,s_r,s_r\beta_r)\mid
s_i\in T_{2i}\}.
\end{eqnarray*}
Suppose that $x\in T_1$. Then there are  $t_i\in T_{2i-1}$
and $s_i\in T_{2i}$ such that $(x,1,\ldots,1)$ equals
\begin{equation*}
(t_1^{-1},t_1^{-1}\alpha_1,t_2^{-1},t_2^{-1}\alpha_2,\ldots,
t_r^{-1},t_r^{-1}\alpha_r,1)\cdot
(1,s_1,s_1\beta_1,s_2,s_2\beta_2,\ldots,s_r,s_r\beta_r).
\end{equation*}
Comparing the entries from  the $k$-th to the $1$-st in order, we obtain that $s_i=t_i=1$ for all
$i$, which is a contradiction if $x\neq 1$. Thus Theorem~\ref{nostripfact}
is proved.
\end{proof}

The following example shows, for a group $G$ that admits uniform
automorphisms with some additional properties, that the direct product
$G^k$ may admit factorisations with subgroups that involve longer strips.

\begin{example}\label{g6example}
  Consider a group $G$. We wish to factorise $G^6$ as $G^6=XY$ where
  $X$ is a
  direct product of two strips of length three and $Y$ is a direct product of
  three strips of length two. Suppose
  that there exist uniform automorphisms $\alpha_2,\ \alpha_3$  of $G$
  such that for all $y_1,\ y_2\in G$ there exists $t$ in $G$ such that
  both $t^{-1}(t\alpha_2)=y_1$ and $t^{-1}(t\alpha_3)=y_2$. In other words,
  the map $G\rightarrow G\times G$ defined by $t\mapsto (t^{-1}(t\alpha_2),
  t^{-1}(t\alpha_3))$ is surjective. If $G$ is a non-trivial finite group, then
  $|G\times G|>|G|$, and hence such automorphisms do not exist for finite
  $G$.
  
Set
\begin{eqnarray*}
X &=& \{(t,t,t,s,s,s) \mid t,\ s \in G \};\\
Y &=& \{(t_1,t_2,t_3,t_1,t_2\alpha_2,t_3\alpha_3) \mid t_i \in G\}.
\end{eqnarray*}

We claim that $G^6=XY$. Indeed, let
$(x_1,\ldots,x_6) \in G^6$. Choose $t$ in $G$ such that
$t(t^{-1}\alpha_2) = x_1x_4^{-1}x_5(x_2^{-1}\alpha_2)$
and
$t(t^{-1}\alpha_3) = x_1x_4^{-1}x_6(x_3^{-1}\alpha_3)$.
Let $t_1\in G$ such that $tt_1=x_1$. Then
it follows by the assumptions above that
\begin{equation*}
(t,t,t,x_4t_1^{-1},x_4t_1^{-1},x_4t_1^{-1})\cdot
(t_1,t^{-1}x_2,t^{-1}x_3,t_1,(t^{-1}x_2)\alpha_2,(t^{-1}x_3)\alpha_3)=(x_1,\ldots,x_6).
\end{equation*}
Therefore $G^6=XY$.
  \end{example}

At the time of writing, we do not know if there exists an infinite simple group admitting
a pair of automorphisms as in Example~\ref{g6example},
and hence we state the following problem. 

\begin{problem}
  Exhibit a group $G$ that admits a pair $(\alpha,\beta)$ of
  automorphisms   such that the map $G\rightarrow
  G\times G$ defined by $g\mapsto(g^{-1}(g\alpha),g^{-1}(g\beta))$ is surjective;
  or prove that no such group exists.
  \end{problem}

\section{Abstract cartesian factorisations involving strips}\label{4}


A characteristically simple group $M$ is said to be {\em finitely completely
  reducible} (\FCR)  if it is
isomorphic to a the direct product $T^k$ for a simple group $T$
and for an integer $k\geq 1$.
A finite characteristically simple group is  FCR.
Suppose that $M$ is a group and $\K=\{K_1,\ldots,
K_\ell\}$ is a family of proper subgroups of $M$. Then $\K$ is said
to be an {\em abstract cartesian factorisation} of $M$ if
\begin{equation}\label{cf}
  M=K_i\left(\bigcap_{j\neq i}K_j\right)\quad\mbox{for all}\quad i\in\{1,\ldots,\ell\};
\end{equation}
the subgroups $K_i$ of an abstract cartesian factorisation
are called {\em cartesian factors} of $M$.

Cartesian factorisations were introduced in~\cite{cs}, where
they were called `cartesian systems of subgroups', to
characterise cartesian decompositions preserved by innately transitive groups.
In this section we investigate cartesian factorisations 
in characteristically simple \FCR-groups in which the cartesian factors
involve strips. Let us introduce notation for this section.
Let $M=T_1\times\cdots\times T_k$ be a  
characteristically simple
\FCR-group with simple normal subgroups $T_1,\ldots,T_k$, and let $G_0$ be a subgroup of $\aut M$ such that the natural
action of $G_0$ on $T_1,\ldots,T_k$ is transitive.
Suppose, in addition,   that $\K=\{K_1,\ldots,K_\ell\}$ is an
abstract $G_0$-invariant cartesian factorisation of
$M$. Set $M_0=\bigcap_i K_i$. If $T$ is abelian, then $T$ is cyclic of prime
order $p$, and $M$ can be considered
as a finite-dimensional vector space over $\F_p$. In this case a cartesian
factorisation of $M$ is essentially a direct sum decomposition. Therefore
we assume that $M$ is non-abelian.

Recall from Section~\ref{set:factstrip} that if
 $I$ is a subset of $\T=\{T_1,\ldots,T_k\}$ or
a subset of $\{1,\ldots,k\}$, then $\sigma_I$ denotes the corresponding
coordinate projection from $M$ onto either $\prod_{T_i\in I}T_i$ or onto
$\prod_{i\in I}T_i$, respectively.

A strip
$X$ is {\em involved} in a subgroup $K$ of $M$ if  $K=X\times
K\sigma_{\mathcal T\setminus\supp X}$ where $\T=\{T_1,\ldots,T_k\}$. We say that
a strip $X$ is involved in a cartesian factorisation $\K$ if $X$ is involved
in a member $K\in\K$. In this case, \eqref{cf} implies that $X$ is involved
in a unique member of $\K$. Uniform automorphisms were introduced
in Section~\ref{set:factstrip}.
In this section we prove the following theorem.

\begin{proposition}\label{mainstripfact}
Suppose that $\K=\{K_1,\ldots,K_\ell\}$, $M=T_1\times\cdots\times T_k$, 
$M_0$, 
and $G_0$ are as above and let $X_1$ and $X_2$ be two non-trivial full strips
involved in $\K$ such that $\supp X_1\cap\supp X_2\neq\emptyset$. Then
the following both hold:
\begin{enumerate}
\item $T_1$ admits a uniform automorphism;
\item $M_0$ is not a subdirect subgroup of $M$.
\end{enumerate}
\end{proposition}
\begin{proof}
  The proof of this theorem is quite involved, and so we split it into a
  series of claims. Assume that the hypotheses of Proposition~\ref{mainstripfact} hold.
  Assume, moreover, that  either $T_1$ does not admit a uniform automorphism, or
  that $M_0$ is a subdirect subgroup of $M$.

\medskip
\noindent {\em Claim 1.}  $|\supp X_1\cap \supp X_2|=1$.

\medskip

\noindent {\em Proof of Claim 1.}
By the definition of `being involved' for  strips,  if $X_1$, $X_2$ are involved in the same
$K_i$ then
they are disjoint as strips. Thus we may assume without loss of generality
that $X_1$ is involved
in $K_{1}$ and $X_2$ is involved in $K_{2}$.
Suppose to
the contrary that $T_{1},\
T_{2}\in\supp X_1\cap \supp X_2$ and set $\sigma=\sigma_{\{1,2\}}$. Then
$Y_1=K_1\sigma$ and
$Y_2=K_2\sigma$ are non-trivial full 
strips in $T_1\times T_2$ such that $Y_1Y_2=M\sigma=T_1\times T_2$.
By Lemma~\ref{orthstrip}, $T_1$ admits
a uniform automorphism. Thus, by our own assumption, 
 $M_0$ is a subdirect subgroup of $M$. Then
$K_{1}\cap K_{2}$ is also a subdirect subgroup of $M$, and so 
$(K_{1}\cap K_{2})\sigma$ is a subdirect subgroup in $T_{1}\times T_{2}$.
Now $Y_1=\{(t,t\alpha)\mid t\in T_1\}$ and
$Y_2=\{(t,t\beta)\mid t\in T_1\}$ where $\alpha,\ \beta:T_1\rightarrow T_2$
are isomorphisms.
Then
$$
(K_{1}\cap K_{2})\sigma\leq Y_1\cap Y_2=\{(t,t\alpha)\mid t\in T_1\mbox{ such that }t\alpha=t\beta\}.
$$
Now the fact that $(K_1\cap K_2)\sigma$ is
subdirect in $T_1\times T_2$ implies that
$t\alpha=t\beta$ for all $t\in T_1$, and hence $\alpha=\beta$. However, this
implies that $Y_1=Y_2$, yielding $Y_1Y_2\neq T_1\times T_2$.
This contradiction proves the claim. \hfill$\Box$

\medskip

\noindent {\em Claim 2.}
  There exists a sequence of strips $X_1,\ldots,X_a$, with $a\geq 3$,
  involved in $\K$ such that
  \begin{enumerate}
  \item $X_1$ and $X_j$ are disjoint for $j\not\in\{a,2\}$;
  \item for $i=2,\ldots,a-1$ and $j\not\in \{i-1,i+1\}$, the strips
    $X_i$ and $X_j$ are disjoint;
  \item $X_a$ and $X_j$ are disjoint for $j\not\in\{a-1,1\}$;
  \item and finally,  
  \begin{multline*}
  |\supp{X_1}\cap \supp{X_2}|=
  |\supp{X_2}\cap \supp{X_3}|=\cdots\\=|\supp{X_{a-1}}\cap \supp{X_a}|=
  |\supp{X_a}\cap \supp{X_1}|=1.
  \end{multline*}
\end{enumerate}

\medskip

\noindent{\em Proof of Claim 2.}
By Claim~1, $\supp{X_1}\cap \supp{X_2}=\{T_t\}$ for some $t\leq k$.
Choose $g\in G_0$ such that $T_t^g\in \supp{X_2}\setminus
\supp{X_1}$; such an element $g$ exists since $G_0$ is transitive on
$\T=\{T_1,\ldots,T_k\}$ and $\supp{X_2}\setminus \supp{X_1}$ is non-empty.
Now $G_0$ acts by conjugation on the set of full strips involved in $\K$,
and so both $X_1^g$ and $X_2^g$ are  full strips involved in $\K$. As
$T_t^g$ is in both $\supp{X_1^g}$ and $\supp{X_2^g}$, but is not in
$\supp{X_1}$ we deduce that there exists a non-trivial  strip $X_3$ in $\K$
distinct from $X_1$, $X_2$ such that $\supp{X_3}\cap\left(\supp{X_2}\setminus
\supp{X_1}\right)\not=\emptyset$
(namely, we can take $X_3$ to be one of $X_1^g$ or $X_2^g$ as at least one
of these is distinct from $X_1$ and $X_2$).
Proceeding in this way we construct a sequence
$X_1,X_2,\ldots$ of distinct, non-trivial strips in $\K$ such
that $\supp{X_{d+1}}\cap\left(\supp{X_d}\setminus
\supp{X_{d-1}}\right)\not=\emptyset$
for each $d\geq 2$.
Since $k$ is finite, there exists $a$ such that 
\[\supp{X_a}\cap\bigl(\supp{X_1}\cup\cdots\cup
\supp{X_{a-2}}\bigr)\not=\emptyset. \]
Let $a$ be the least integer such that this property holds.
The conditions imposed on $X_1$ and $X_2$ imply that $a\geq 3$. 
By removing some initial segment of the sequence and relabelling
the $X_i$ if
necessary,
we may assume that the intersection $\supp{X_a}\cap \supp{X_1}$ is
non-empty,
while $\supp{X_a}\cap \supp{X_d}=\emptyset$ if $2\leq d \leq a-2$ for some $a\geq 3$. Now, applying
Claim~1 a number of times, the sequence $X_1,\ldots,X_a$ is as required.\hfill$\Box$

\medskip

Now assume that the conditions of Claim~2 are valid,
$X_1$ and $X_2$ are non-disjoint strips involved in $\K$ and
select $X_1,\ldots,X_a$ as in the proof of Claim~2.
By relabelling the $K_i$ we may assume that $X_1$ is involved in $K_1$.
Let $1=i_1<i_2<\cdots<i_d<a$ be such that among the $X_i$ 
the strips $X_{i_1},\ldots,X_{i_d}$ are precisely the ones
that are involved in $K_1$. 
Note that $X_a$ is not involved in $K_1$ since $\supp{X_a}$ and $\supp{X_1}$ are not
disjoint. Also, $i_{j+1}\geq i_j+2$ for all
$j=1,\ldots,d-1$
since $\supp{X_{i_j}}$ and $\supp{X_{i_j+1}}$ are not disjoint. We
may also relabel the $T_i$ so that
$$
\{T_1\}=\supp{X_a}\cap
\supp{X_1}\quad\mbox{and}\quad\{T_2\}=\supp{X_1}\cap \supp{X_2},
$$
and so that for $j=2,\ldots,d$,
\begin{eqnarray*}
\{T_{2j-1}\}&=&\supp{X_{i_{j}-1}}\cap \supp{X_{i_j}};\\
\{T_{2j}\}&=&\supp{X_{i_j}}\cap \supp{X_{i_{j}+1}}.
\end{eqnarray*}
It follows from Claim~2, that $T_1,\ldots,T_{2d}$ are
pairwise distinct.
Define  the projection map
\begin{equation}\label{sigmakeq}
  \sigma :M\to
T_1\times\cdots\times T_{2d}\quad\mbox{and}\quad
\widehat{K_1}=\bigcap_{i\neq 1}K_i.
\end{equation}

\medskip

\noindent{\em Claim 3.}
Using the notation introduced above, the following hold.
\begin{enumerate}
\item $K_1\sigma$ is a direct product  
 $Y_1\times\cdots\times Y_d$ such that each $Y_i$ is a non-trivial
full strip in $T_{2i-1}\times T_{2i}$ and $Y_i=\{(t,t\alpha_i)\mid t\in T_{2i-1}
\}$
for some isomorphism $\alpha_i:T_{2i-1}\rightarrow T_{2i}$. 
\item $\widehat K_1\leq Z_1\times\cdots\times Z_d$, such that,
for $i=1,\ldots,d-1$, $Z_i=\{(t,t\beta_i)\mid t\in T_{2i}\}$ is a non-trivial
full strip in $T_{2i}\times T_{2i+1}$ where $\beta_i:T_{2i}\rightarrow T_{2i+1}$
is an isomorphism.
Further
$Z_d=\{(t\beta_d,t)\mid t\in T_{2d}\}$ is a non-trivial full strip in 
$T_{2d}\times T_{1}$ where $\beta_d:T_{2d}\rightarrow T_1$ is an isomorphism.
\end{enumerate}

\medskip

\noindent{\em Proof of Claim 3.}
Assertion~(1) follows from the observation that
$$
K_1\sigma=(X_{i_1}\times\cdots\times X_{i_d})\sigma=Y_1\times\cdots\times Y_d
\quad
\mbox{with}\quad \mbox{$Y_j=X_{i_j}\sigma$ for each $j$}.
$$

Let us prove assertion~(2). It suffices to show that
$\widehat{K_1}\sigma_{\{2i,2i+1\}}\leq Z_i$ for $i=1,\ldots,d-1$ and 
$\widehat{K_1}\sigma_{\{1,2d\}}\leq Z_d$. We prove the claim for 
$i=1$, that is, for 
$\widehat{K_1}\sigma_{\{2,3\}}$, noting that
the proof for the other projections is identical.
Set $r=i_2-1$. Then the strips $X_2,\ldots,X_r$ are `between' $X_1$ and
$X_{i_2}$ in the sequence of the $X_i$ and they are  not involved in
$K_1$. 
Choose  $T_{m_1},\ldots,T_{m_r}$ such that $\{T_{m_i}\}=\supp X_i\cap 
\supp X_{i+1}$. By the choice made earlier, we have  $T_{m_1}=T_2$ and 
$T_{m_r}=T_3$. 
Let $\sigma'$ 
denote the projection onto $T_{m_1}\times\cdots\times T_{m_{r}}$. By 
Claim~2, 
the indices $m_1,\ldots,m_{r}$ are pairwise distinct.
Suppose that $x=(t_{m_1},\ldots,t_{m_{r}})$
is an element of $\widehat{K_1}\sigma'$ with $t_i\in T_{m_i}$.
The strip $X_2$ is involved in  $K_{m}$ for some  $m\neq 1$. Recall
that $X_2$ covers $T_{m_1}=T_2$ and $T_{m_2}$.
Then $x\in K_{m}\sigma'$ and hence $(t_{m_1},t_{m_2})\in X_2\sigma_{\{m_1,m_2\}}$.
Thus there exists an isomorphism $\gamma_{m_1}:T_{m_1}\rightarrow T_{m_2}$ such that
$X_2\sigma_{\{m_1,m_2\}}=\{(t,t\gamma_{m_1})\mid t\in T_{m_1}\}$. 
Using the same argument, we find that there exist
isomorphisms
$\gamma_{m_i}:T_{m_i}\rightarrow T_{m_{i+1}}$, for 
$i=1,\ldots,r-1$, such that
$X_{i+1}\sigma_{\{m_i,m_{i+1}\}}=\{(t,t\gamma_{m_2})\mid t\in T_{m_i}\}$.
Thus $t_{m_r}=t_{m_1}\gamma_{m_1}\cdots\gamma_{m_{r-1}}$.
Set $\beta_1=\gamma_{m_1}\cdots\gamma_{m_{r-1}}$. This argument
shows that
$$
\widehat{K_1}\sigma_{\{2,3\}}\leq\{(t,t\beta_1)\mid t\in T_2\}:=Z_1.
$$
Therefore  $\widehat{K_1}\sigma_{\{2,3\}}\leq Z_1$. The proof for the other projections
is identical.
This shows that assertion~(2) holds.\hfill$\Box$

\medskip

\noindent{\em Claim 4.}
Use the notation of Claim~3, and set $\alpha=\alpha_1\beta_1\alpha_2\cdots\alpha_d\beta_d$. Then the following hold.
\begin{enumerate}
\item $\alpha\in\aut T_1$ and $\alpha$ is uniform.
\item  $M_0$ is not a subdirect subgroup of $M$ where $M_0=\bigcap_i K_i$.
\end{enumerate}

\medskip

\noindent{\em Proof of Claim 4.}
  (1) It follows from Claim~3 that $\alpha\in\aut T_1$. 
Since $K_1\widehat{K_1}=M$, we have 
  $(K_1\sigma)(\widehat{K_1}\sigma)=T_1\times\cdots\times
T_{2d}$ with $\sigma$ and $\widehat K_1$ as in~\eqref{sigmakeq}. Therefore
\begin{equation}\label{yzfact}
  (Y_1\times\cdots\times Y_d)(Z_1\times\cdots\times Z_d)=
  T_1\times\cdots\times T_{2d}.
\end{equation}
Since the factorisation in~\eqref{yzfact} is as in
Lemma~\ref{doublestrips}, it follows from Lemma~\ref{doublestrips} that
$\alpha$ is uniform.

(2) By definition, $K_1\cap \widehat{K_1}=M_0$. Suppose that $M_0$ is subdirect in $M$. Then 
$M_0\sigma$ is also a subdirect subgroup of $T_1\times\cdots\times T_{2d}$. 
Suppose that $x=(t_1,t_2,\ldots,t_{2d})\in M_0\sigma$. Then 
$x\in K_1\sigma$ and $x\in\widehat{K_1}\sigma$, and so 
$t_{2i}=t_{2i-1}\alpha_i$, for $i=1,\ldots,d$, and also $t_{2i+1}=t_{2i}\beta_i$ for
$i=1,\ldots,d-1$, and $t_1=t_d\beta_d$. Thus $t_1=t_1\alpha_1\beta_1\cdots\alpha_d\beta_d=t_1\alpha$. Since
$M_0\sigma$ is subdirect, this has to hold for all $t_1\in T_1$, and hence
$\alpha=1$. However, by part~(1), $\alpha$ is uniform, which is a 
contradiction, as the identity automorphism is not uniform.
\hfill$\Box$

\medskip

Now Proposition~\ref{mainstripfact} follows at once from Claim~4.
\end{proof}

\section{Quasiprimitive groups of diagonal type}\label{section:sd}

Recall that a permutation group acting on $\Omega$ is quasiprimitive if
all the non-trivial normal subgroups of $G$ are transitive.
A quasiprimitive group $G$ is said to be of {\em diagonal
  type} if it has a unique minimal normal \FCR\ subgroup $M$ such that
$M=T_1\times\cdots\times T_k$ where the $T_i$ are non-abelian simple groups and
a point stabiliser $M_\alpha$ is a subdirect subgroup of $M$; that is,
denoting the coordinate projections by $\sigma_i$,
$M_\alpha\sigma_i=T_i$ for all $i$. By Scott's Lemma~\cite[Lemma, page~328]{sco:rep},
we have, in such a quasiprimitive group of diagonal type, that
$M_\alpha$ is a direct product of pairwise disjoint full strips\footnote{Scott's Lemma is most often applied to finite simple groups, but
  it holds for infinite simple groups and the proof presented in~\cite{sco:rep}
   does not assume finiteness.}.
The group $G$ is said to have {\em simple diagonal type} if
$M_\alpha$ is simple (that is, $M_\alpha$ is a full strip),
and otherwise $G$ is said to have {\em compound diagonal type}.

If $\Gamma$ is a set and $\ell\geq 2$,
then the wreath product $\sym\Gamma\wr\sy\ell$ can be considered
as a permutation group on $\Gamma^\ell$ in its product action, which
is defined as
$$
(\gamma_1,\ldots,\gamma_\ell)(g_1,\ldots,g_\ell)\sigma=
(\gamma_{1\sigma^{-1}}g_{1\sigma^{-1}},\ldots,\gamma_{\ell\sigma^{-1}}g_{\ell\sigma^{-1}})
$$
for all $\gamma_1,\ldots,\gamma_\ell\in\Gamma$,
$g_1,\ldots,g_\ell\in\sym\Gamma$, and $\sigma\in\sy\ell$. 

\begin{theorem}\label{diagonal}
Let $G$ be a quasiprimitive group  of diagonal
type acting on $\Omega$ with minimal normal subgroup
$M=T_1\times\cdots\times T_k$ where the $T_i$ are non-abelian simple groups.
Then $G$ can be embedded into a subgroup $W$ of $\Omega$ that is
permutationally isomorphic to $\sym\Gamma\wr\sy\ell$
in product action with $|\Gamma|\geq 2$ and
$\ell\geq 2$ if and only if $G$ is of compound diagonal type.
\end{theorem}
\begin{proof} We suppose that $G$ is quasiprimitive of diagonal type with a minimal normal subgroup $M$ as in the statement.
  Suppose that there exists a subgroup $W$ of $\sym\Omega$
  that is permutationally isomorphic to $\sym\Gamma\wr\sy\ell$
  with $|\Gamma|\geq 2$ and $\ell\geq 2$ such that $G\leq W$. Then
  $\Omega$ can be identified with $\Gamma^\ell$, and so from now
  on we assume that $\Omega=\Gamma^\ell$ and that $W=\sym\Gamma\wr\sy\ell$.
  Note that $M$ is
  transitive on $\Omega$, since $G$ is quasiprimitive.

  Let $\pi:W\rightarrow\sy\ell$
  denote the natural projection. Since $M$ is a minimal normal
  subgroup of $G$,  either $M\leq \ker\pi$ or $M\cap\ker\pi=1$.
  In the latter case, the restriction of $\pi$ to $M$ is a faithful
  permutation representation of $M$ and hence $M$ is isomorphic to a
  subgroup of $\sy\ell$. This implies that $M$ is finite and, since
  $M$ is transitive on $\Omega$, the sets $\Omega$ and
  $\Gamma$ are finite. Now if $p$ is a prime dividing $|\Gamma|$, then
  $p^\ell\mid |\Omega|$, and hence, as $M$ is transitive on $\Omega$,
  $p^\ell\mid|M|$, which gives
  $p^\ell\mid\ell!$. Since this is
  impossible (see~\cite[Exercise~2.20]{jones}),  we must have
   $M\leq\ker\pi$.
  Therefore
  $M$ is contained in the base group $B=(\sym\Gamma)^\ell$ of $W$.
  
  The action of the base group can be viewed via $\ell$ projection
  maps
  $\pi_1,\ldots,\pi_\ell:B\rightarrow\sym\Gamma$ given by
  $(g_1,\ldots,g_\ell)\pi_i=g_i$ for 
  $(g_1,\ldots,g_\ell)\in(\sym\Gamma)^\ell$. In this way we may write,
  for $g\in B$ and $(\gamma_1,\ldots,\gamma_\ell)\in\Omega$, that
  $$
  (\gamma_1,\ldots,\gamma_\ell)g=
  (\gamma_1(g\pi_1),\ldots,\gamma_\ell (g\pi_\ell)).
  $$
  In particular, the last equation is valid if $g\in M$.
  We consider the homomorphisms $\pi_1,\ldots,\pi_\ell$ as permutation
  representations of $B$.
  Choose $\gamma\in\Gamma$, set
  $\omega=(\gamma,\ldots,\gamma)\in\Omega$, and, for
  $i=1,\ldots,\ell$, let $K_i$ denote the stabiliser in $M$ 
  of $\gamma$ under $\pi_i$ .
  Since $M$ is transitive on $\Omega$, each $M\pi_i$ is transitive on $\Gamma$,
  and, since $|\Gamma|\geq 2$, each $K_i$ is a proper subgroup of $M$. 

  We claim that the set $\K=\{K_1,\ldots,K_\ell\}$
  of proper subgroups of $M$  is an abstract cartesian
  factorisation of $M$. Let us prove that
  $M=K_1\left(\bigcap_{i\neq 1} K_i\right)$;
  the other factorisations in equation~\eqref{cf} can be shown
  similarly. Setting $\bar{K_1}=\bigcap_{i\neq 1} K_i$ and
  noting that $K_1$ is the point stabiliser in $M$ under its transitive
  action on $\Gamma$ by $\pi_1$, the
  factorisation $M=K_1\bar{K_1}$ is equivalent to the assertion
  that $\bar{K_1}$ is transitive on $\Gamma$ under the representation
  $\pi_1$. Suppose that $\gamma'\in\Gamma$ and consider the point
  $\omega'=(\gamma',\gamma,\ldots,\gamma)$. Since $M$ is transitive on
  $\Omega$, there exists some $m\in M$ such that $\omega m=\omega'$.
  Considering the definitions of $\omega$ and $\omega'$, this implies
  that $\gamma(m\pi_1)=\gamma'$ and $\gamma (m\pi_i)=\gamma$ for all
  $i=2,\ldots,\ell$; that is, $m\in\bar{K_1}$. Since $\gamma'$ is
  chosen arbitrarily, $\bar{K_1}$ is transitive on $\Gamma$ under
  $\pi_1$, and so $K_1\bar{K_1}=M$, as claimed. As noted above,
  the other factorisations of~\eqref{cf} follow similarly, and
  $\K$ is an abstract cartesian factorisation for $M$, as claimed.

  Since $\bigcap_{i\geq 1}K_i$ is the intersections of the stabilisers of
  $\gamma$ under the representations $\pi_1,\ldots,\pi_\ell$ and
  $\omega=(\gamma,\ldots,\gamma)$, it follows that
  $$
  \bigcap_{i=1}^\ell K_i=M_\omega.
  $$
  Further, $G_\omega$ permutes by conjugation the $\ell$ direct
  factors of the base group $B$, and so the set $\K$ is invariant under conjugation by $G_\omega$. 

  Since $M$ is a minimal normal subgroup of $G$, $G$ is transitive by conjugation on the $T_i$, and so the $T_i$  are pairwise
  isomorphic. Let
$T$ denote the common isomorphism type of the $T_i$. Since $G$ has diagonal type,
 $M_\omega$ is a subdirect subgroup of $M$.
For each $K\in\K$ we have $M_\omega\leq K$, and
so 
all elements of $\K$ are subdirect subgroups of $M$.
Let
$K_1,\ K_2\in \K$ be  distinct subgroups. Then $K_1,\
K_2\neq M$, and so,
$K_1,\ K_2$ involve non-trivial full strips $X_1$ and
$X_2$, say, and, by the factorisation in equation~\eqref{cf}, $M=K_1K_2$, which
implies 
$X_1\neq X_2$.
By Proposition~\ref{mainstripfact}, $X_1$ and $X_2$ are disjoint strips.
This means, in particular, that if $X$ is a non-trivial full strip involved
in $K_j$ covering $T_i$, then $T_i\leq K_{m}$ for all $m\neq j$. 
Therefore if $X_1,\ldots,X_s$ are the non-trivial full strips involved
in $\K$, then $M_\omega=X_1\times\cdots\times X_s$. Further, since
each $K\in\K$ involve at least one non-trivial full strip and $|\K|\geq 2$,
the argument
above shows that $s\geq 2$.
Hence $M_\omega$ is a direct product of at least
two non-trivial full strips, and so $G$ has compound diagonal type.

Conversely, suppose that
$G$ has compound diagonal type, and so $M_\omega$ is a direct product
$M_\omega=X_1\times\cdots\times X_r$
of non-trivial full strips $X_i$.
Then $M$ admits a non-trivial direct product
decomposition $M=M_1\times \cdots\times M_r$ such that
$$
M_\omega=(M_1\cap M_\omega)\times\cdots\times (M_r\cap M_\omega).
$$
Setting $\Delta$ to be the right coset space
$[M_1:M_1\cap M_\omega]$, the transitivity of $M$ on $\Omega$
allows us to identify $\Omega$ with the direct power $\Delta^r$ and,
under this identification, $G$ becomes a subgroup of $\sym\Delta\wr\sy r$ in
product action.
\end{proof}

Theorem~\ref{mainth} is an easy consequence of  Theorem~\ref{diagonal}.

\def\cprime{$'$}

\end{document}